\documentclass[a4paper,12pt]{article}

\usepackage{amsthm}
\usepackage{amsmath}
\usepackage{amssymb}
\usepackage{mathrsfs}
\usepackage{graphicx}
\usepackage{hyperref}
\usepackage{bm}
\usepackage{dsfont}
\usepackage{enumitem}

\usepackage[english]{babel}
\usepackage[T1]{fontenc}
\usepackage[utf8]{inputenc}

\usepackage[all]{xy}
\usepackage{caption}
\usepackage{graphicx}
\usepackage{tikz}
\usepackage{tikz-cd}
\usetikzlibrary{patterns}
\usetikzlibrary{arrows}
\usetikzlibrary{arrows.meta}
\usepackage{comment}
\usepackage{authblk}

\usepackage{geometry}
\usepackage[square,numbers]{natbib}
\newcommand{\pdim}{\mathrm{pdim}}
\newcommand{\ext}{\mathrm{Ext}}
\newcommand{\tor}{\mathrm{Tor}}

\newcommand{\HH}{\mathrm{HH}}

\newcommand{\cd}{\mathrm{cd}}

\newcommand{\Ker}{\mathrm{Ker}}
\newcommand{\id}{\mathrm{id}}

\newtheorem{thm}{Theorem}[section]

\newtheorem{prop}[thm]{Proposition}

\theoremstyle{remark}
\newtheorem{rem}[thm]{Remark}

\newtheorem{ex}[thm]{Example}

\theoremstyle{definition}
\newtheorem{defi}[thm]{Definition}

\newtheorem*{nota}{Notations}

\begin{document}

\title{Homological smoothness of Hopf-Galois extensions}
\date{}
\author{Julian Le Clainche\thanks{Université Clermont Auvergne, CNRS, LMBP, F-63000 CLERMONT-FERRAND, FRANCE}}

\maketitle

\begin{abstract}
We show that if  $H$ is a Hopf algebra with bijective antipode and $B\subset A$ is a faithfully flat $H$-Galois extension, then $A$ is homologically smooth  if $H$ and $B$ are.    
\end{abstract}

\section{Introduction}

An algebra $A$ is said to be homologically smooth if $A$ admits, as an $A$-bimodule, a finite length resolution by finitely generated projective $A$-bimodules. Homological smoothness is an appropriate analogue of regularity for (noncommutative) algebras, and is the basic condition involved in homological duality questions for algebras, see \cite{bieri_groups_1973,van_den_bergh_relation_1998,van_den_bergh_erratum_2002,ginzburg_calabi-yau_nodate}. The aim of this paper is to provide a result that produces new examples of homologically smooth algebras, in the setting of Hopf-Galois extensions, the analogue of principal bundles in noncommutative algebra. 

Recall \cite{KT81} that if $H$ is a Hopf algebra and $A$ is a $H$-comodule algebra, the algebra extension $A^{\text{co}H} \subset A$ is said to be an $H$-Galois extension if a certain canonical map $\beta : A \otimes_B A \to A\otimes H$ is bijective, see Section \ref{Hopf-Galois}.
This framework includes examples of various nature such as Hopf crossed products \cite{montgomery_hopf_1993} or exact sequences of Hopf algebras \cite{andruskiewitsch_extensions_1996} (in particular exact sequences of groups).

In this setting, our main result is as follows:

\begin{thm}\label{th smoothness}
Let $H$ be a Hopf algebra with bijective antipode and let $B\subset A$ be an $H$-Galois extension such that $A$ is faithfully flat as left and right $B$-module. If $H$ and $B$ are homologically smooth algebras, then $A$ is homologically smooth as well. 
\end{thm}

Similar results were known to be true in some  particular cases, notably in the setting of twisted Calabi-Yau algebras: the case of Galois objects (i.e. for $B=k$) was done by Yu \cite{yu_hopf-galois_2016} and smash products of algebras by Hopf algebras were studied by Le Meur \cite{le_meur_smash_2019}.

The main ingredients to prove Theorem \ref{th smoothness} are
\begin{enumerate}
    \item Stefan's spectral sequences \cite{stefan_hochschild_1995} for Hopf-Galois extensions; 
    \item a smoothness criterion using ${\tor}$ due to Bieri-Eckmann \cite{bieri_homological_1983}, which seems to have been slightly forgotten in the recent literature.
\end{enumerate}

The paper is organized as follows. Section 2 consists of preliminaries about homological finiteness of modules and reminders about Hochschild cohomology. In section 3 we recall the definition of Hopf-Galois extensions and we introduce Stefan's spectral sequences. Section 4 is devoted to the proof of Theorem \ref{th smoothness} and in the final Section 5, we present an illustrating example of Theorem \ref{th smoothness}.

\begin{nota}
Throughout the paper, we work over a field $k$, and all algebras are (unital) $k$-algebras.
If $A$ is an algebra, the opposite algebra is denoted $A^{op}$ and the enveloping algebra $A\otimes A^{op}$ is denoted $A^e$.
The category of right (resp. left) $A$-modules is denoted $\mathcal{M}_{A}$ (resp. ${}_A \mathcal{M}$).
We have $\mathcal M_A={_{A^{op}}\mathcal{M}}$, hence any definition we give for left modules has an obvious analogue for right modules, that we will not necessarily give.
An $A$-bimodule structure is equivalent to a left (or right) $A^e$-module structure. Indeed, if $M$ is an $A$-bimodule, a left (resp. right) $A^e$-module structure on $M$ is defined by $$(a\otimes b) \cdot m = amb \text{ (resp. } m\cdot (a\otimes b) = bma\text{ )} \quad \mathrm{for} \quad a,b\in A \quad \mathrm{and}\quad m\in M.$$
The  category of $A$-bimodule is thus identified with the category $\mathcal{M}_{A^e}$, or with the category ${}_{A^e}\mathcal{M}$.

If $M$ and $N$ are left $A$-modules, $\ext$ spaces (\cite{weibel_introduction_2008}) are denoted $\ext_A^\bullet (M,N)$ and if $P$ is a right $A$-module $\tor$ spaces are denoted $\tor^A_\bullet(P,M)$. 
We will also need to consider $\ext$ spaces in categories of right modules, which we denote $\ext^\bullet_{A^{op}}(-,-)$.

 

If $H$ is a Hopf algebra, its comultiplication, counit and antipode are  denoted $\Delta, ~ \varepsilon$ and $S$ and we will use Sweedler notation in the usual way, i.e. for $h \in H$, we write $\Delta(h)=h_{(1)}\otimes h_{(2)}$. See \cite{montgomery_hopf_1993}.
\end{nota}

\section{Homological preliminaries}

In this section we recall the various homological ingredients  needed in the paper. 

\begin{defi}
    Let $A$ be an algebra. The \textit{projective dimension} of an $A$-module $M$ is defined by 
        \[\pdim_A(M) := \min\left\{ n \in \mathbb{N},~ M \text{ admits a length} \ n \ \text{resolution by projective } A\text{-modules} \right\} \in \mathbb{N}\cup \{\infty\}.\]
\end{defi}

The projective dimension can as well be characterized by
\begin{align*}
\pdim_A(M) &= \min\left\{ n \in \mathbb{N},~ \ext_A^{n+1}(M,N)=\{0\} \text{ for any $A$-module $N$}  \right\}  \\
& = \max \left\{ n \in \mathbb{N},~ \ext^{n}_A(M,N)\not=\{0\} \text{ for some $A$-module $N$}.  \right\} \
\end{align*}



We now recall various finiteness conditions on modules.

\begin{defi}
    Let $A$ be an algebra and let $M$ be an $A$-module. 
    \begin{enumerate}
        \item The $A$-module $M$ is said to be \textit{of type $FP_\infty$} if it admits a projective resolution $P_\bullet \to M$ with $P_i$ finitely generated for all $i$. 
        \item The $A$-module $M$ is said to be \textit{of type FP} if it admits a finitely generated projective resolution of finite length.
    \end{enumerate}
\end{defi}

The following result characterizes modules of type $FP$ among those of type $FP_\infty$, see e.g \cite[Chapter VIII]{brown_cohomology_1982}.

\begin{prop}\label{prop:FP-FPinfty}
Let $A$ be an algebra. An $A$-module $M$ is of type $FP$ if and only if it is of type $FP_\infty$ and $\pdim_A(M)$ is finite. In this case we have     
\[\pdim_A(M) = \max\left\{ n \in \mathbb{N},~ \ext_A^{n}(M,A)\not=\{0\}   \right\}\]
\end{prop}

The following result is \cite[Corollary 1.6]{bieri_homological_1983}. The implication (iii)$\Rightarrow$(i) is due to Bieri-Eckmann \cite{Bieri_Eckmann74}, and  gives an effective condition to check that a module is of type $FP_\infty$ and  will be useful in section \ref{Smoothness}.

\begin{prop}\label{characterization FP}
Let $A$ be an algebra and $M$ a left (resp. right) $A$-module. The following conditions are equivalent: 
\begin{enumerate}[label= (\roman*)]
    \item The left (resp. right) $A$-module $M$ is of type $FP_\infty$;,
    \item For any direct product $\prod_{i\in I}N_i$ of right (resp. left) $A$-modules, the natural map  $\tor_k^A(\prod_{i \in I}N_i,M)\to \prod_{i\in I}\tor_k^A(N_i,M)$ (resp. $\tor_k^A(M,\prod_{i \in I}N_i)\to \prod_{i\in I}\tor_k^A(M,N_i)$ is an isomorphism for all $k\geq0$;
    \item For any direct product $\prod A$ of arbitrary many copies of $A$, we have $\tor_i^{A}(\prod A, M)= 0$ (resp. $\tor_i^{A}(N, \prod A)= 0$) for $i\geq 1$ and the natural map $\left(\prod A \right) \otimes_{A} M \to \prod M$(resp. $M \otimes_{A} \left(\prod A\right) \to \prod M$) is an isomorphism.
\end{enumerate}
\end{prop}


There are similar results for $\ext$ functors, but will only use the following particular instance.

\begin{prop}\label{FP ext}
    Let $A$ be an algebra and $M$ a right (resp. left) $A$-module, if $M$ is of type $FP_\infty$ then $\ext_{A^{op}}(M,-)$ (resp. $\ext_A(M,-)$) commutes with direct sums.
\end{prop}


Those above general homological finiteness notions specify to the case of bimodules which is our case of interest when considering Hochschild (co)homology.

\begin{defi}\label{defi : Hochschild}
    Let $A$ be an algebra.
\begin{enumerate}
    \item The \textit{cohomological dimension} of $A$ is defined to be
$\cd(A) := \pdim_{A^e}(A)$.
\item The algebra $A$ is said to be \textit{homologically smooth} if $A$ is of type $FP$ as a left $A^e$-module.
\item Let $M$ an $A$-bimodule. The \textit{Hochschild cohomology} spaces of $A$ with coefficients in $M$ are the vector spaces $\HH^\bullet(A,M) := \ext_{A^e}^\bullet(A,M)$. The \textit{Hochschild homology} spaces of $A$ with coefficients in $M$ are the vector spaces $\HH_\bullet(A,M) := \tor^{A^e}_\bullet(M,A)$.
\end{enumerate}    
\end{defi}

Notice that if $A$ is a homologically smooth algebra, then $\cd(A) = \max \left\{n\in \mathbb{N},~ \HH^n(A,A^e) \neq 0\right\}$, by Proposition \ref{prop:FP-FPinfty}.



We finally record that if $A=H$ is a Hopf algebra, homological smoothness can be expressed using $H$-modules rather than $H$-bimodules. 

\begin{thm}\cite[Proposition A.2]{wang_calabi-yau_2017}\label{duality Hopf}
     Let $H$ be a Hopf algebra. The following assertions are equivalent: 
     \begin{enumerate}[label= (\roman*)]
         \item The algebra $H$ is homologically smooth;
         \item The right $H$-module $k_\varepsilon$ is of type $FP$;
         \item The left $H$-module $_\varepsilon k$ is of type $FP$.
     \end{enumerate}
\end{thm}

\section{Hopf-Galois extensions}\label{Hopf-Galois}


\subsection{Definitions and examples}\label{HG def}

Recall that if $H$ is a Hopf algebra, an $H$-comodule algebra is an algebra $A$ together with an algebra  map $\rho : A \to A \otimes H$ making it into an $H$-comodule. 

\begin{defi}
Let $H$ be a Hopf algebra. An $H$-comodule algebra $A$ is said to be \textit{an $H$-Galois extension} of $B := A^{\text{co}H} = \left\{a\in A,~ a_{(0)}\otimes a_{(1)} = a \otimes 1\right\}$ if the (canonical) map 
\begin{align*}
    \beta :~  A\otimes_B A &\longrightarrow A \otimes H\\
    a \otimes_B b &\longmapsto ab_{(0)}\otimes b_{(1)}
\end{align*}
is bijective. An $H$-Galois extension of the base field $k$ is called \textit{an $H$-Galois object}.
\end{defi}



An overview of the theory of Hopf-Galois extensions can be found in \cite{Schau2004}. We now list a number of important examples.

\begin{ex}
Let $H$ be a Hopf algebra. Then $H$ is an $H$-Galois object, with the right $H$-comodule structure given by the comultiplication.
\end{ex}

\begin{ex}
    Let $G$ be a group and let $A$ a $G$-graded algebra, i.e. an algebra with a decomposition $A = \bigoplus\limits_{g\in G} A_g$ such that $A_g A_h \subset A_{gh}$ for any $g,h \in G$ and $1 \in A_e$, where $e$ denotes the neutral element of $G$. The $G$-grading makes $A$ into a $kG$-comodule algebra, and $A_e \subset A$ is a $kG$-Galois extension if and only if $A$ is $G$-strongly graded, which means that $A_g A_h = A_{gh}$ for any $g,h \in G$.

    In that case, the inverse of $\beta$ is given by $$\beta^{-1}(a\otimes g) = \sum\limits_i a c_i \otimes_{A_e} d_i \text{ for } a\in A, g\in G$$
    where $c_i \in A_{g^{-1}}$ and $d_i\in A_g$ are elements such that $\sum\limits_i c_i d_i = 1$.
\end{ex}

\begin{ex}
Let $H$ be a Hopf algebra and let $B$ be a left $H$-module algebra, i.e. $B$ is a left $H$-module  with $h\cdot(ab) = (h_{(1)}\cdot a)(h_{(2)}\cdot b)$ and $h.1 = \varepsilon(h)1$ for any $h\in H$, $a,b\in $B. Let $A$ be the smash product algebra $B\# H$. Then $A$ is an $H$-Galois extension of $B$ where the $H$-comodule structure on $A$ is given by 
$$a\# h \mapsto (a \# h_{(1)}) \otimes h_{(2)}$$
and the inverse of $\beta$ is 
\begin{align*}
   \beta^{-1}  : ~ (B\# H) \otimes H &\longrightarrow  (B\# H)\otimes_B (B\# H)\\
   (a\#h)\otimes k & \longmapsto  (a \# h S(k_{(1))}) \otimes_B (1\# k_{(2)})
\end{align*}
\end{ex}

\begin{ex} \label{exact sequence}
    Let  $p : A \to H$ be a surjective Hopf algebra map. Then $p$ induces a right $H$-comodule algebra structure on $A$, given by $a \mapsto a_{(1)} \otimes p(a_{(2)})$. Let $B= A^{\text{co}H}$. If $B^+A = {\rm Ker}(p)$, then $B\subset A$ is an $H$-Galois extension, with the inverse of the canonical map $\beta$ being given by 
\begin{align*}
\beta^{-1} : ~ A \otimes H  & \longrightarrow  A \otimes_B A\\ 
a\otimes p(a') & \longmapsto a S(a'_{(1)}) \otimes_B a'_{(2)}
\end{align*}
Recall \cite{andruskiewitsch_extensions_1996} that a  sequence of Hopf algebra maps 
\[k \longrightarrow B \overset{i}{\longrightarrow} A \overset{p}{\longrightarrow} H \longrightarrow k\]
is said to be exact  if the following conditions hold:
\begin{enumerate}[label = (\arabic*), parsep=0cm,itemsep=0.1cm,topsep=0cm]
    \item $i$ is injective and $p$ is surjective,
    \item $\Ker(p) = i(B)^+ A = A i(B)^+$,
    \item $i(B) = A^{co H} = {}^{co H}A$,
\end{enumerate}
where $A^{co H}= \{a \in A,~ a_{(1)}\otimes p(a_{(2)}) = a\otimes 1\}$ and ${}^{co H}A= \{a \in A,~ p(a_{(1)})\otimes a_{(2)} = 1\otimes a\}$.
Exact sequences of groups correspond to exact sequences of their group algebras.

Hence, if $k \longrightarrow B \overset{i}{\longrightarrow} A \overset{p}{\longrightarrow} H \longrightarrow k$ is an exact sequence of Hopf algebras, then by the previous consideration  we have that $A$ is an $H$-Galois extension of $i(B)$ where the $H$-comodule structure on $A$ is given by $a \mapsto a_{(1)} \otimes p(a_{(2)})$. 
\end{ex}

\begin{rem}
Let $B\subset A$ be an $H$-Galois extension. Since the canonical map is left $A$-linear, its inverse is uniquely determined by $\kappa : H\to A\otimes_BA$ defined by $\kappa(h)=\beta^{-1}(1\otimes h)$. Denoting $\kappa(h)=h^{\langle 1\rangle}\otimes_B h^{\langle 2\rangle}$, the following identities, which we will not use directly but are essential for some constructions in the next subsection, are \cite[Remark 3.4]{schneider_representation_1990}: 
\begin{align*}
    &b h^{\langle 1\rangle}\otimes_B h^{\langle 2\rangle} = h^{\langle 1\rangle}\otimes_B h^{\langle 2\rangle} b, \quad 
    a_{(0)}a_{(1)}^{\langle1\rangle}\otimes_B a_{(1)}^{\langle 2 \rangle} = 1\otimes_B a, \quad
    h^{\langle 1\rangle}h^{\langle 2\rangle} = \varepsilon(h) \\
    &h^{\langle 1\rangle}\otimes_B (h^{\langle 2\rangle})_{(0)}\otimes (h^{\langle 2\rangle})_{(1)} =  h_{(1)}^{\langle 1\rangle}\otimes_B h_{(1)}^{\langle 2\rangle} \otimes h_{(2)}, \\
    & (h^{\langle 1\rangle})_{(0)}\otimes_B h^{\langle 2\rangle} \otimes (h^{\langle1\rangle})_{(1)} = h_{(2)}^{\langle 1\rangle}\otimes_B h_{(2)}^{\langle 2\rangle} \otimes S(h_{(1)})\\
    &(hk)^{\langle 1\rangle}\otimes_B (hk)^{\langle 2\rangle} =  k^{\langle 1\rangle}h^{\langle 1\rangle}\otimes_B h^{\langle 2\rangle} k^{\langle 2\rangle}
\end{align*}
\end{rem}

\subsection{The Stefan spectral sequences}\label{HG Stefan}

The fundamental tool that we will use in the proof of Theorem \ref{th smoothness} is the Stefan spectral sequence, constructed in \cite{stefan_hochschild_1995} by using the Grothendieck spectral sequence. 

It involves a right (resp. left) $H$-module structures on the cohomology spaces $\HH^q(B,M)$ (resp. on the homology spaces $\HH_q(B,M)$) for $q\geq0$ and any $A$-bimodule $M$ defined in \cite{stefan_hochschild_1995}. For $q=0$, the $H$-action on $\HH^0(B,M)\simeq M^B$ is given by $m\cdot h = h^{\langle 1\rangle} m h^{\langle 2 \rangle}$ for $m\in M^B$ and $h\in H$ and it is extended to $q\geq 0$ using the machinery of cohomological functors \cite[Theorem 7.5]{brown_cohomology_1982} . In \cite{garcia_iglesias_computation_2022}, explicit formulas are given for $q\geq 0$ in the case of a smash product. 

For $q= 0$, the $H$-action on $\HH_0(B,M)\simeq M/[M,B]$ is given by $h\cdot \pi_M(m) = \pi_M\left(h^{\langle 2\rangle} m h^{\langle 1 \rangle}\right)$ for $m\in M$ and $h\in H$, where $\pi_M$ denotes the natural projection of $M$ onto $M/[M,B]$.

\begin{thm}[\cite{stefan_hochschild_1995}]\label{stefan}
Let $H$ be a Hopf algebra, let $B\subset A$ be an $H$-Galois extension and let $M$ be an $A$-bimodule.

\begin{enumerate}[label =(\roman*)]
    \item Assume that $A$ is flat as left and right $B$-module. Then there is a spectral sequence 
\[E^{p,q}_2 = \ext_{H^{op}}^p(k_\varepsilon,\HH^q(B,M)) \implies \HH^{p+q}(A,M)\]
which is natural in M.

    \item Assume that $A$ is projective as left and right $B$-module. Then there is a spectral sequence 
\[E_{p,q}^2 = \tor^H_p(k_\varepsilon,\HH_q(B,M)) \implies \HH_{p+q}(A,M)\]
which is natural in M.
\end{enumerate}
\end{thm}

As explained in \cite{stefan_hochschild_1995}, the above spectral sequences are natural generalizations of the usual  Lyndon-Hochschild-Serre spectral sequences in group or Lie algebra (co)homology.

\begin{rem}\label{ff=proj}
It is shown in \cite[Theorem 4.10]{schauenburg_generalized_2005}  that if $H$ is a Hopf algebra with bijective antipode, and $B\subset A$ an $H$-Galois extension with $A$ faithfully flat as a left and right $B$-module, then $A$ is projective as a left and right $B$-module. Thus, in order to invoke one of the above spectral sequences,  we can indifferently assume projectivity or faithful flatness for an $H$-Galois extension $B\subset A$.
\end{rem}

\begin{rem}
    In \cite{stefan_hochschild_1995}, the homology spaces $\HH_\bullet(A,M)$ are defined as the spaces $\tor_\bullet^{A^e}(A,M)$ for $A$ an algebra and $M$ an $A$-bimodule. This is equivalent to Definition \ref{defi : Hochschild}, indeed, if $P$ is an $A$-bimodule, we have $M\otimes_{A^e} P\simeq P\otimes_{A^e} M$ hence, the complexes defining $\tor_\bullet^{A^e}(M,A)$ and $\tor_\bullet^{A^e}(A,M)$ are isomorphic, hence their homology are and  $\tor_\bullet^{A^e}(M,A) \simeq \tor_\bullet^{A^e}(A,M)$.
\end{rem}


\section{Smoothness of Hopf-Galois extensions}\label{Smoothness}

This section is dedicated to the proof of our main theorem. We begin with a result on finiteness of the cohomological dimension.

\begin{prop}\label{prop:finitecd}
Let $H$ be a Hopf algebra with bijective antipode and let $B\subset A$ be an $H$-Galois extension such that $A$ is flat as left and right $B$-module.   We have $\cd(A)\leq \cd(B)+\cd(H)$, and hence if $\cd(B)$ and $\cd(H)$ are finite, so is $\cd(A)$.
\end{prop}

\begin{proof}
Under those assumptions, we can use Theorem \ref{stefan}. Thus, for every $A$-bimodule $M$, we get a spectral sequence $$E^{pq}_2 = \ext_{H^{op}}^p(k_\varepsilon,\HH^q(B,M)) \implies \HH^{p+q}(A,M).$$
If $\cd(B)$ or $\cd(H)$ is infinite, there is nothing to show, so we assume that these are finite.
 For $p > \cd(H)$ or $q> \cd(B)$, we have $E^{pq}_2 = \{0\} $. Now, if we denote $d_2^{p,q} : E^{pq}_2 \to E^{p+2,q-1}_2$ the differential on the second page of the spectral sequence, we get that for $p> \cd(H) $ or $q>\cd(B)$ the maps $d_2^{p,q} : E^{p,q}_2 \to E^{p+2,q-1}_2$ and $d_2^{p-2,q+1} : E^{p-2,q+1}_2 \to E_2^{p,q}$ are both $0$ hence $\{0\} = E_2^{p,q} \simeq E_3^{p,q} \simeq ..\simeq \HH^{p+q}(A,M)$. We obtain that $\HH^{n}(A,M)=\{0\}$
 for $n>\cd(B)+\cd(H)$, and hence $\cd(A) \leq \cd(H)+\cd(B)$.    
\end{proof}


\begin{proof}[Proof of Theorem \ref{th smoothness}]
Let $H$ be a Hopf algebra with bijective antipode and let $B\subset A$ be an $H$-Galois extension. We assume that $A$ is faithfully flat as left and right $B$-module and  that $H$ and $B$ are homologically smooth algebras.
We know from Proposition \ref{prop:finitecd} that $A$ has finite cohomological dimension, hence to prove that $A$ is homologicallly smooth, 
it remains to prove, by Proposition \ref{prop:FP-FPinfty}, that $A$ is of type $FP_\infty$ as an $A$-bimodule. 

To prove that $A$ is of type $FP_\infty$ as an $A$-bimodule, we will use the characterization of Proposition \ref{characterization FP}, and hence we consider the $A$-bimodule $M = \prod A^e$ where $\prod$ is an arbitrary direct product. We have to show that for $n\geq 1$, one has $\HH_{n}(A,M)=0$ and that $\HH_{0}(A,M)\simeq \tor_0^{A^e}(M,A) \simeq \left(\prod A^e\right) \otimes_{A^e} A \simeq \prod A \simeq \prod \tor_0^{A^e}(A^e,A) \simeq \prod \HH_0(A,A^e)$. 


Under the hypothesis that $A$ is faithfully flat as a right and left $B$-module and regarding Remark \ref{ff=proj}, we can use Theorem \ref{stefan} (ii). Thus, for every $A$-bimodule $M$ we get a spectral sequence $$E_{pq}^2 = \tor^H_p(k_\varepsilon,\HH_q(B,M)) \implies \HH_{p+q}(A,M).$$
The algebra $B$ is of type $FP_\infty$ as a left $B^e$-module hence $\HH_q(B,-) = \tor_q^{B^e}(-,B)$ commutes with direct products. Moreover, $A^e$ is flat as a left and right $B^e$-module \cite[Lemma 2.1]{stefan_hochschild_1995} thus for $q\geq1$, we get that $$\HH_q(B,\prod A^e) \simeq \prod \HH_q(B,A^e) = \{0\}.$$
Hence we have $E_{pq}^2=\{0\}$ for $q>0$.

If $\pi_M$ (resp. $\pi_{A^e}$) denotes the natural projection of $M$ (resp. $A^e$) onto $\HH_0(B,M) \simeq M/[M,B]$ (resp. $\HH_0(B,A^e) \simeq A^e/[A^e,B]$), the natural isomorphism $\mu : \HH_0(B,M) \to \prod \HH_0(B,A^e)$ is given by $\mu(\pi_M\bigl((x_i \otimes y_i)_i\bigr)) = \bigl(\pi_{A^e}(x_i\otimes y_i)\bigr)_i$ for any $x_i,y_i \in A$. Hence, for $h\in H$, we have $$\mu(h\cdot \pi_M\bigl((x_i \otimes y_i)_i\bigr)) = \mu(\pi_M\bigl((h^{\langle2 \rangle}x_i \otimes y_ih^{\langle 1\rangle})_i \bigr)) = \bigl(\pi_{A^e}(h^{\langle 2\rangle}x_i\otimes y_ih^{\langle 1\rangle})\bigr)_i = h\cdot\biggl(\bigl(\pi_{A^e}(x_i\otimes y_i)\bigr)_i\biggr).$$ 
The natural isomorphism $\HH_0(B,M) \simeq \prod \HH_0(B,A^e)$ is thus an isomorphism of $H$-modules. Therefore, since the algebras $B$ and $H$ are homologically smooth, we get, using Proposition \ref{characterization FP}, for $p \geq 0$ 
$$E_{p0}^2=\tor_p^H(k_\varepsilon,\HH_0(B,M)) \simeq \tor_p^H(k_\varepsilon,\prod\HH_0(B,A^e)) \simeq \prod \tor_p^H(k_\varepsilon,\HH_0(B,A^e)).$$
The $A^e$-module $A^e$ is projective, hence \cite[Proposition 4.4]{stefan_hochschild_1995} ensures that \[\tor_p^H(k_\varepsilon,\HH_0(B,A^e))=\{0\} \ \text{for $p\geq 1$.}\] Hence for $p\geq 1$ we have
$$E_{p,0}^2=\tor_p^H(k_\varepsilon,\HH_0(B,M)) \simeq \prod \tor_p^H(k_\varepsilon,\HH_0(B,A^e)) = \{0\}$$ 
We get that $E^2_{p,q}= \tor^H_p(k_\varepsilon,\HH_q(B,M)) = \{0\}$ for $(p,q) \neq (0,0)$, and hence the spectral sequence ensures that $\HH_{n}(A,M)= \{0\}$ for $n>0$.


Finally, using \cite[Proposition 4.2]{stefan_hochschild_1995}, that  $k_\varepsilon$ is of type $FP_\infty$ as a right $H$-module and that $B$ is smooth, we obtain 
\[
\begin{aligned} \HH_0(A, \prod A^e) &\simeq \tor_0^H(k_\varepsilon,\HH_0(B,\prod A^e))\\ &\simeq \tor^H_0(k_\varepsilon,\prod\HH_0(B,A^e))\\ &\simeq \prod\tor^H_0(k_\varepsilon,\HH_0(B,A^e))\\ &\simeq \prod \HH_0(A,A^e).
\end{aligned}
\]
It is not difficult to check that the above isomorphism is the natural map $\mu$ in the third item of Proposition \ref{characterization FP}, the first an last isomorphism being explicit from \cite[Proposition 4.2]{stefan_hochschild_1995}, and the second and third one being obtained from the natural respective maps $\mu$ as well. 
 We thus conclude from Proposition \ref{characterization FP} that $A$ is homologically smooth.
\end{proof}

\begin{rem}
    In the situation of Theorem \ref{th smoothness}, we have seen that $\cd(A)\leq \cd(B)+\cd(H)$. We have not been able to show that the equality $\cd(A) = \cd(B) +\cd(H)$ holds in general, although we suspect it does. It holds in the case of Galois objects \cite{yu_hopf-galois_2016}, of smash products \cite{le_meur_smash_2019}, and in the case of exact sequences of Hopf algebras, as we will show in a forthcoming paper.
\end{rem}

\section{An example}

We finish the paper by presenting an illustrative example.

\begin{defi}
Let $B$ be a commutative $k$-algebra, let $b\in B$ and let $q\in k^\times$ with $q^2 \neq 1$. The $k$-algebra $U_q^{B,b}$ is defined by \[U_q^{B,b} = B\left\langle g,\ g^{-1},\ e,\ f \ \big\vert \ eg = q^{-2} ge,\ gf = q^{-2}fg, \ ef -fe = bg -\frac{1}{q-q^{-1}}g^{-1}\right\rangle\]   
\end{defi}

When $B=k$ and $b=(q+q^{-1})^{-1}$, the algebra $U_q^{B,b}$ is the quantized enveloping algebra $U_q(\mathfrak{sl}_2)$.

\begin{prop}
    If  $B$ is an homologically smooth  commutative  algebra, then $U_q^{B,b}$ is as well homologically smooth, with $\cd(U_q^{B,b}) \leq \cd(B)+3$.
\end{prop}

\begin{proof}
Recall that denoting $E=e$, $K=g$ and $F=f$, the algebra $U_q(\mathfrak{sl}_2)$ has a
Hopf algebra structure defined by 
\[\Delta(E)= 1\otimes E + E\otimes K,\ \Delta(F) = K^{-1}\otimes F + F\otimes 1, \ \Delta(K)= K\otimes K,\]  \[S(E) = -EK^{-1},\ S(F) = -KK,\ S(K) =K^{-1},\ \varepsilon(E)=\varepsilon(F) = 0,\ \varepsilon(K)=1.\]
More generally, $U_q^{B,b}$ is a right $U_q(\mathfrak{sl}_2)$-comodule algebra with $\rho(a)=a\otimes 1$ for any $a\in B$, $\rho(g) = g \otimes K,\ \rho(e) = e \otimes K + 1 \otimes E,\ \rho(f) = f \otimes 1 + g^{-1} \otimes F$. It is shown in \cite[Lemma 16]{gunther_crossed_1999} that $B\subset U_q^{B,b}$ is a cleft (hence in particular free) $U_q(\mathfrak{sl}_2)$-Galois extension. Hence combining \cite[Proposition 3.2.1]{chemla_rigid_2004} (for $\mathfrak{g} = \mathfrak{sl_2}$), Proposition \ref{prop:finitecd} and Theorem \ref{th smoothness}, we obtain the announced result.
\end{proof}

\bibliography{references_smooth}

\begin{thebibliography}{10}

\bibitem{andruskiewitsch_extensions_1996}
{\sc Andruskiewitsch, N., and Devoto, J.}
\newblock Extensions of {Hopf} algebras.
\newblock {\em St. Petersburg Math. J. 7\/} (1996), 17--52.

\bibitem{bieri_homological_1983}
{\sc Bieri, R.}
\newblock {\em Homological dimension of discrete groups}.
\newblock Queen Mary College, 1983.

\bibitem{bieri_groups_1973}
{\sc Bieri, R., and Eckmann, B.}
\newblock Groups with homological duality.
\newblock {\em Invent. Math. 20}, 2 (1973), 103--124.

\bibitem{Bieri_Eckmann74}
{\sc Bieri, R., and Eckmann, B.}
\newblock Finiteness properties of duality groups.
\newblock {\em Comment. Math. Helv. 49\/} (1974), 74--83.

\bibitem{brown_cohomology_1982}
{\sc Brown, K.~S.}
\newblock {\em Cohomology of groups}, vol.~87 of {\em Graduate Texts in Mathematics}.
\newblock Springer-Verlag, 1982.

\bibitem{chemla_rigid_2004}
{\sc Chemla, S.}
\newblock Rigid dualizing complex for quantum enveloping algebras and algebras of generalized differential operators.
\newblock {\em J. Algebra 276}, 1 (2004), 80--102.

\bibitem{garcia_iglesias_computation_2022}
{\sc García~Iglesias, A., and Sánchez, J.}
\newblock On the computation of {Hopf} 2-cocycles with an example of diagonal type.
\newblock {\em Glasg. Math. J. 1965}, 1 (2023), 141--169.

\bibitem{ginzburg_calabi-yau_nodate}
{\sc Ginzburg, V.}
\newblock Calabi-{Yau} algebras.
\newblock arXiv:math/0612139.

\bibitem{gunther_crossed_1999}
{\sc Günther, R.}
\newblock Crossed products for pointed hopf algebras.
\newblock {\em Comm. Algebra 27}, 9 (1999), 4389--4410.

\bibitem{KT81}
{\sc Kreimer, H.~F., and Takeuchi, M.}
\newblock Hopf algebras and {G}alois extensions of an algebra.
\newblock {\em Indiana Univ. Math. J. 30}, 5 (1981), 675--692.

\bibitem{le_meur_smash_2019}
{\sc Le~Meur, P.}
\newblock Smash products of {Calabi}–{Yau} algebras by {Hopf} algebras.
\newblock {\em J. Noncommut. Geom. 13}, 3 (2019), 887--961.

\bibitem{montgomery_hopf_1993}
{\sc Montgomery, S.}
\newblock {\em Hopf algebras and their actions on rings}, vol.~82 of {\em CBMS Regional Conference Series in Mathematics}.
\newblock American Mathematical Society, Providence, RI, 1993.

\bibitem{Schau2004}
{\sc Schauenburg, P.}
\newblock Hopf-{G}alois and bi-{G}alois extensions.
\newblock In {\em Galois theory, {H}opf algebras, and semiabelian categories}, vol.~43 of {\em Fields Inst. Commun.} Amer. Math. Soc., Providence, RI, 2004, pp.~469--515.

\bibitem{schauenburg_generalized_2005}
{\sc Schauenburg, P., and Schneider, H.}
\newblock On generalized {Hopf} galois extensions.
\newblock {\em J. Pure Appl. Algebra 202}, 1-3 (2005), 168--194.

\bibitem{schneider_representation_1990}
{\sc Schneider, H.}
\newblock Representation theory of {Hopf} galois extensions.
\newblock {\em Israel J. Math. 72}, 1-2 (1990), 196--231.

\bibitem{stefan_hochschild_1995}
{\sc Stefan, D.}
\newblock Hochschild cohomology on {Hopf} {Galois} extensions.
\newblock {\em J. Pure Appl. Algebra 103}, 2 (1995), 221--233.

\bibitem{van_den_bergh_relation_1998}
{\sc Van Den~Bergh, M.}
\newblock A relation between {Hochschild} homology and cohomology for gorenstein rings.
\newblock {\em Proc. Amer. Math. Soc. 126}, 5 (1998), 1345--1348.

\bibitem{van_den_bergh_erratum_2002}
{\sc Van Den~Bergh, M.}
\newblock Erratum to : "{A} relation between {Hochschild} homology and cohomology for gorenstein rings".
\newblock {\em Proc. Amer. Math. Soc 130}, no 8 (2002), 2809--2810.

\bibitem{wang_calabi-yau_2017}
{\sc Wang, X., Yu, X., and Zhang, Y.}
\newblock Calabi-{Yau} property under monoidal {Morita}-{Takeuchi} equivalence.
\newblock {\em Pac. J. Math. 290}, 2 (2017), 481--510.

\bibitem{weibel_introduction_2008}
{\sc Weibel, C.}
\newblock {\em Introduction to homological algebra}.
\newblock Cambridge University Press, 2008.

\bibitem{yu_hopf-galois_2016}
{\sc Yu, X.}
\newblock Hopf-{G}alois objects of {C}alabi-{Y}au {H}opf algebras.
\newblock {\em J. Algebra Appl. 15}, 10 (2016), 1650194, 19.

\end{thebibliography}

\vspace{1cm}

Université Clermont Auvergne, CNRS, LMBP, F-63000 CLERMONT-FERRAND, FRANCE

\end{document}